\documentclass[11pt,twoside,letter]{article}
\usepackage[dvips]{graphicx}
\usepackage{fancyhdr}

\usepackage{amsmath}    % need for subequations
\usepackage{graphicx}   % need for figures
\usepackage{verbatim}   % useful for program listings
\usepackage{amssymb}
\usepackage{amsthm}
\usepackage{mathrsfs}
\usepackage{bbm}
\usepackage{exscale}
\usepackage{enumerate}
\newtheorem{theorem}{Theorem}[section]

\newtheorem{proposition}[theorem]{Proposition}
\usepackage{color}                    % For creating coloured text and background

\theoremstyle{remark}
\newtheorem{remark}[theorem]{Remark}
\newtheorem{question}[theorem]{Question}

\theoremstyle{definition}
\newtheorem{definition}[theorem]{Definition}

\newcommand{\finsum}[3]{
\underset{#1=#2}{\overset{#3}\sum}}

\newcommand{\dint}{
\displaystyle\int}

\newcommand{\dsum}{
\displaystyle\sum}

\newcommand{\inflim}[1]{
\underset{#1\to\infty}\lim}

\newcommand{\zsum}[1]{
\underset{#1\in\Z}\dsum}

\newcommand{\nsum}[1]{
\underset{#1\in\N}\dsum}

\newcommand{\R}{
\mathbb{R}}

\newcommand{\N}{
\mathbb{N}}
\newcommand{\Z}{
\mathbb{Z}}
\newcommand{\F}{
\mathscr{F}}
\newcommand{\supp}{
\textnormal{supp}}

\newcommand{\sinc}{
\textnormal{sinc}}

\newcommand{\bracket}[1]{
\left\langle#1\right\rangle}

\newcommand{\I}{
\mathscr{I}}

\newcommand{\phica}{
\widehat{\phi_{\alpha,c}}}

\newcommand{\Lachat}{
\widehat{L_{\alpha,c}}}

 \newcommand{\eps}{
 \varepsilon}

\begin{document}

%%%%%%%%%%%%%%%%%%%%%%%%%%%%%%%%%%%%%%%%%%%%%%%%%%
%
% The Title
%
\title{\bf\vspace{-39pt} Stability and Robustness of RBF Interpolation}

%%%%%%%%%%%%%%%%%%%%%%%%%%%%%%%%%%%%%%%%%%%%%%%%%%
%
% The Authors
%

\author{Jean-Luc Bouchot\thanks{J.-L.B. was partially supported by the European Research
Council through the grant StG 258926 . It was finalized while J.-L.B. was staying at the Hausdorff Institute of Mathematics attending the Hausdorff Trimester Program on Mathematics of Signal Processing.} \\ \small Chair for Mathematics C (Analysis), RWTH Aachen University, Pontdriesch 10 \\
\small 52062 Aachen, Germany \\ \small bouchot@mathc.rwth-aachen.de\\
\\
Keaton Hamm \thanks{K.H. would like to especially thank the organizers of the 2015 SampTA conference where the collaboration that lead to this article began.}\\ \small Department of Mathematics, Vanderbilt University, 1326 Stevenson Center \\
\small Nashville, TX 37240, USA  \\ \small keaton.hamm@vanderbilt.edu}

%%%%%%%%%%%%%%%%%%%%%%%%%%%%%%%%%%%%%%%%%%%%%%%%%%
%
% Do not print the date
%
\date{}

%%%%%%%%%%%%%%%%%%%%%%%%%%%%%%%%%%%%%%%%%%%%%%%%%%
%
% Make the title and set the header styles to be
%   fancy for this page.  STSIP will adjust these
%   headings later.
%
\maketitle
%\thispagestyle{fancy}

%%%%%%%%%%%%%%%%%%%%%%%%%%%%%%%%%%%%%%%%%%%%%%%%%%
%
% Setup the Headings for the article
%   - Please use all caps and initials for your
%     first and middle names
%
\markboth{\footnotesize \rm \hfill J.-L. BOUCHOT AND K. HAMM \hfill}
{\footnotesize \rm \hfill STABLE RBF INTERPOLATION\hfill}

%%%%%%%%%%%%%%%%%%%%%%%%%%%%%%%%%%%%%%%%%%%%%%%%%%
%
% The Abstract, keywords, and phrases
%

\begin{abstract}
This article considers how some methods of uniform and nonuniform interpolation by translates of radial basis functions -- focusing on the case of the so-called general multiquadrics -- perform in the presence of certain types of noise.  These techniques provide some new avenues for interpolation on bounded domains by using fast Fourier transform methods to approximate cardinal functions associated with the RBF.
\vspace{5mm} 
\vspace{3mm}
\end{abstract}

%%%%%%%%%%%%%%%%%%%%%%%%%%%%%%%%%%%%%%%%%%%%%%%%%%
%
% Finally, you can write your article
%

%%%%%%%%%%%%%%%%%%%%%%%%%%%%%%%%%%%%%%%%%%%%%%%%%
%%%%%%  Introduction                        %%%%%
%%%%%%%%%%%%%%%%%%%%%%%%%%%%%%%%%%%%%%%%%%%%%%%%%
\section{Introduction}

The classical sampling problem may be asked in two parts: first, for a given class of signals, does it suffice to know the samples, or values, of a signal at a given discrete set of points in order to recover the signal in some manner?  Second, how might the signals be recovered, and moreover, how might it be done in a computationally efficient way?  There are many theoretical and practical answers to this problem in various settings, and perhaps the most fundamental result is the classical Whittaker--Kotelnikov--Shannon sampling theorem \cite{Shannon}, which states that $L_2(\R)$ functions whose Fourier transform is supported on $[-\pi,\pi]$, for example, may be recovered in $L_2$ and uniformly via 
$$f(x) = \sum_{j\in\Z}f(j)\dfrac{\sin(\pi (x-j))}{\pi(x-j)}.$$
While Whittaker \cite{Whittaker} saw the series above as a {\em cardinal interpolation series}, i.e. evaluating the right-hand side at $k\in\Z$ produces $f(k)$, it was later shown that the convergence was uniform for bandlimited signals.  

The drawback of this sampling formula for practical considerations is that the series is difficult to approximate well by truncation since the cardinal sine function $\sinc(x):=\sin(\pi x)/(\pi x)$ decays slowly (like $|x|^{-1}$).  There is an abundance of literature tracing back to this fundamental theorem, and correspondingly many techniques to get around the slow decay of sinc.  One such method is oversampling, which can be costly in practice.  Another method intimately related to the analysis here is what I. J. Schoenberg, the father of spline theory, terms {\em summability methods} (see \cite{Schoenberg} and references therein).  Specifically, one attempts to replace sinc in the series above by another function which decays more rapidly, nonetheless requiring that the new series is close to the original signal in whatever way one wants to measure (e.g. in $L_2$ or uniformly).

Some study of summability methods using cardinal functions formed from translates of a single {\em radial basis function} (RBF) -- one which satisfies $\phi(x)=\phi(|x|)$-- has been made  \cite{Baxter,Buhmann,BuhmannBook,HammLedford,HammLedford2,HMNW,LedfordCardinal,RS4,Siva}.  Cardinal functions are those which satisfy the interpolatory condition $L(k)=\delta_{0,k}$, $k\in\Z$, of which sinc is the prototypical example.  Such cardinal functions fashioned from radial basis functions have a special form in the Fourier transform domain as discussed in the sequel.  

Building on these results, there are many techniques for sampling at nonuniform sets in $\R^d$.  Of course, the analysis is typically much simpler in one dimension, whereas many of the techniques that are currently known in higher dimensions rely on the geometry of the points in $\R^d$ in a nontrivial way.  Even the first part of the classical sampling question leaves some deep open questions in this area and has seen links with many interesting realms of mathematics including space-tiling, convex geometry, basis theory, and abstract harmonic analysis.  Of interest to this work are those nonuniform sampling methods which use RBFs \cite{BSS,HammJAT,HammZonotopes,LedfordScattered,schsiv}.  For a survey of some of these themes using multiquadrics consult \cite{HammSampTA}, of which this article is a continuation.

The primary concern of this article is twofold: to analyze what happens to various interpolation schemes involving RBFs for either bandlimited or time-limited signals in the presence of noise, and also to give some indication of the computational feasibility and methodology for performing the sampling scheme. We consider two kinds of noise and the effect they have on the sampling and reconstruction of certain classes of signals, with the primary method being that of sampling via interpolation.  Consequently, this allows us to make some comparisons with the traditional literature on RBF interpolation on compact domains.  The error estimates for noisy interpolation also allow for a stability result for interpolation on closed intervals, essentially stating that up to a constant, it suffices to consider interpolation at uniform points on the interval.  For the sake of ease, we present the results in one dimension, and discuss which ones lend themselves to multivariate analogues in the final section.

\begin{comment}The primary contribution here is to analyze what happens to various interpolation schemes involving RBFs for either bandlimited or time-limited signals in the presence of noise.  We consider two kinds of noise and the effect they have on the sampling and reconstruction of certain classes of signals, and we also give some indication of the computational feasibility and methodology for performing the sampling scheme.  The primary method will be that of sampling via interpolation, and consequently, this allows us to compare our method with the traditional literature on RBF interpolation on compact domains.\end{comment}

The rest of the paper is laid out as follows: we begin with a section of basic notions, facts, and recall some previous results on multiquadric interpolation of bandlimited signals.  Section \ref{SECNoise} contains new results on how the interpolation methods from Section \ref{SECBasic} behave in the presence of noise.  Then, in Section \ref{SECRates}, approximation rates in terms of the spacing of the samples in one dimension and the effect of noise on them are considered.  Section \ref{SECComputations} discusses how our method compares with classical RBF interpolation theory for compactly supported functions, and we end with a brief discussion of extensions in Section \ref{SECExtensions}.

\begin{comment}The rest of the paper is laid out as follows: we list some basic definitions and facts in Section \ref{SECBasic}, followed by the definition of the interpolation method in Section \ref{SECScheme}.  Section \ref{SECRecovery} recalls the recovery results in this setup for both uniform and nonuniform sampling.  Section \ref{SECNoise} uses the combination of uniform and nonuniform results from the previous section to determine what happens under two distinct types of noise.  Then, in Section \ref{SECRates}, approximation rates in terms of the spacing of the samples in one dimension and the effect of noise on them are considered.  Section \ref{SECComputations} discusses how our method compares with classical RBF interpolation theory for compactly supported functions, and we end with a brief discussion of extensions in Section \ref{SECExtensions}.\end{comment}

%%%%%%%%%%%%%%%%%%%%%%%%%%%%%%%%%%%%%%%%%%%%%%%%%
%%%%%%  Basic                               %%%%%
%%%%%%%%%%%%%%%%%%%%%%%%%%%%%%%%%%%%%%%%%%%%%%%%%

\section{Background Material}\label{SECBasic}

If $\Omega\subset\R$ is a set of positive Lebesgue measure, then let $L_p(\Omega)$ be the typical Banach space of $p$--integrable (or essentially bounded in the case $p=\infty$) Lebesgue measurable functions on $\Omega$ with its usual norm.  If no set is specified, it is to be assumed that $\Omega=\R$.  Similarly, $\ell_p(I)$ are the usual sequence spaces of $p$--summable sequences indexed by a set $I$, and if no set is specified, we mean $\ell_p(\Z)$.  Let $C_0(\R)$ be the space of continuous functions on $\R$ which vanish at infinity.

For $k\in\N$, let $W_p^k(\Omega)$ be the Sobolev space of $L_p(\Omega)$ functions whose derivatives of order up to $k$ are all in $L_p(\Omega)$.  
\begin{comment}If $\alpha=(\alpha_1,\dots,\alpha_d)$ is a multi-index, then let $D^\alpha$ be the derivative operator given by $D^\alpha g = \frac{\partial^{\alpha_1}}{\partial x_1^{\alpha_1}}\dots\frac{\partial^{\alpha_d}}{\partial x_d^{\alpha_d}}g$. \end{comment}
The seminorm on the Sobolev space may be defined by
$$|g|_{W_p^k(\Omega)}:= \left(\dint_\Omega|g^{(k)}(x)|^pdx\right)^\frac{1}{p} = \|g^{(k)}\|_{L_p(\Omega)},$$
and the following is a norm on $W_p^k$: $\|g\|_{W_p^k(\Omega)}:=\|g\|_{L_p(\Omega)}+|g|_{W_p^k(\Omega)}.$  Again, if no set is specified, we refer to $W_p^k(\R)$.

For a function $f\in L_1$, define its Fourier transform via
$$\widehat{f}(\xi):=\dint_{\R}f(x)e^{-ix\xi}dx.$$  Thus under suitable conditions (for example, if $f$ is continuous and $\widehat{f}\in L_1$) the following inversion formula holds: $f(x) = (\widehat{f})^\vee(x) = (2\pi)^{-1}\int_{\R}\widehat{f}(\xi)e^{i\xi x}d\xi.$  The Fourier transform can be uniquely extended to a linear isomorphism of $L_2$ onto itself, and under this normalization, the Parseval/Plancherel Identity states that $2\pi\|f\|_{L_2}=\|\widehat{f}\|_{L_2}.$  

For a parameter $\sigma>0$, define the Paley--Wiener space of bandlimited functions (with band-size $\sigma)$ via
$PW_\sigma:=\{f\in L_2(\R): \widehat{f}=0 \textnormal{ a.e. outside } [-\sigma,\sigma]\}$.  The Paley--Wiener Theorem states that an equivalent definition of the latter is the space of entire functions of exponential type $\sigma>0$ whose restriction to $\R$ is in $L_2$.  As all Paley--Wiener spaces are isometrically isomorphic, we typically restrict ourselves to the canonical space $PW_\pi$; however all of the results mentioned here may be dilated to a space with different band-size.

The interpolation scheme considered in the sequel will use the following ideas for point distributions in $\R$, whose definitions are taken from \cite{yo}.
\begin{definition}\label{DEFCIS}
\begin{enumerate}
    \item[(i)] A sequence $(x_n)_{n\in\N}\subset\R$ is a {\em complete interpolating sequence} (CIS) for $PW_\pi$ provided for every $a\in\ell_2(\N)$, there exists a unique $f\in PW_\pi$ such that $f(x_n)=a_n$, $n\in\N$.
    \item[(ii)] A sequence $(f_n)_{n\in\N}$ in a Hilbert space $\mathcal{H}$ is a {\em Riesz basis} for $\mathcal{H}$ provided $(f_n)$ is complete and there exist $C_1,C_2>0$ such that the following inequality holds for all $a\in\ell_2(\N)$:
    \begin{equation}\label{EQFrameBounds}C_1\|a\|_{\ell_2}^2\leq\left\|\sum_{n\in\N}a_nf_n\right\|_{\mathcal{H}}^2\leq C_2\|a\|_{\ell_2}^2.\end{equation}
\end{enumerate}
\end{definition}

For Paley--Wiener spaces, complete interpolating sequences are equivalent to Riesz bases of exponentials in the corresponding $L_2$ space in the Fourier domain via the following theorem.

\begin{theorem}[\cite{yo}, Theorem 9, p. 143]
$(x_n)_{n\in\N}\subset\R$ is a CIS for $PW_\pi$ if and only if $\left(e^{-ix_n(\cdot)}\right)_{n\in\N}$ is a Riesz basis for $L_2[-\pi,\pi]$.
\end{theorem}

\begin{comment}
It should be noted that the matter of existence of a CIS for the Paley--Wiener space over a given set $S$ is somewhat delicate whenever $d>1$.  For example, it is unknown if such a sequence exists when $S=B_2^d$, the Euclidean ball in $\R^d$.  However, there are some examples of interest, namely finite unions of disjoint intervals in $\R$ \cite{nitzan} or disjoint cubes with parallel axes in $\R^d$ \cite{nitzan2}.  There are some specific types of polytopes that admit Riesz bases of exponentials as well, such as centrally symmetric polytopes whose faces of co-dimension 1 are also centrally symmetric and whose vertices lie on a lattice \cite{GrepstadLev} (see also \cite{Kolountzakis} for a subsequent alternate method of proof).
\end{comment}

For subsequent use, we catalog here some facts related to Riesz bases of exponentials.  First, it bears noting that in dimension 1, such bases are abundant by the following classical result.
\begin{theorem}[Kadec's $1/4$--Theorem, \cite{Kadec}]\label{THMKadecs}
If $(x_j)_{j\in\Z}\subset\R$ satisfies $$\underset{j\in\Z}\sup\;|x_j-j|<\frac{1}{4},$$ then $(x_j)$ is a CIS for $PW_\pi$.  Moreover, the bound of $1/4$ is sharp.
\end{theorem}
There are higher dimensional analogues of Kadec's theorem, for example, see \cite{Bailey1,Bailey2,SunZhou}.  Having a sufficient condition, we also note that a necessary condition for $(x_j)$ to be a CIS is that it is separated, i.e. $\inf_{j\neq k}|x_k-x_j|>0$.  For a complete characterization by Pavlov using Muckenhoupt's $A_p$ condition in terms of zeros of so-called sine-type entire functions, see \cite{Pavlov}.

There are also some important notions of stability of complete interpolating sequences which will be required. 
\begin{comment}\begin{proposition}\label{PROPFinitePerturbation}
If $(x_j)_{j\in\N}$ is a CIS for $PW_S$, and $(y_j)_{j\in\N}$ is such that $y_j\neq x_j$ for only finitely many $j\in\N$, then $(y_j)$ is also a CIS for $PW_S$.
\end{proposition}\end{comment}

\begin{theorem}\label{THMPerturbedCIS}

\begin{enumerate}
\item[(i)]\cite[Theorem 1]{Young_Perturb}  If $(x_j)_{j\in\N}$ is a CIS for $PW_\pi$, then there exists a positive constant $L$ such that if $|y_j-x_j|\leq L$ for every $j\in\N$, then $(y_j)$ is a CIS for $PW_\pi$.
\item[(ii)]\cite[Lemma 3.1]{PakShin} If $(x_j)_{j\in\N}$ is a CIS for $PW_\pi$, and $(y_j)_{j\in\N}$ is such that $y_j\neq x_j$ for only finitely many $j\in\N$, and $y_j \neq y_i$ for any $i \neq j$, then $(y_j)$ is also a CIS for $PW_\pi$.
\end{enumerate}
\end{theorem}

%%%%%%%%%%%%%%%%%%%%%%%%%%%%%%%%%%%%%%%%%%%%%%%%%
%%%%%%  Interpolation Scheme (General)      %%%%%
%%%%%%%%%%%%%%%%%%%%%%%%%%%%%%%%%%%%%%%%%%%%%%%%%

\subsection{The Interpolation Scheme}\label{SECScheme}

The primary concern of this paper is to analyze a scheme which samples a smooth function via interpolation from a shift-invariant space of certain radial basis functions.  To wit, consider the following general problem: given a function $f$ with a certain order of smoothness (e.g. in $PW_\pi$ or $W_p^k(\R)$), a separated sequence $X:=(x_j)_{j\in\Z}\subset\R$, and a radial basis function $\phi:\R\to\R$ such that $\phi(x)=\phi(|x|)$, find an interpolating function of the form
\begin{equation}\label{EQInterpForm}\I_\phi f(x) = \dsum_{j\in\Z}a_j\phi(x-x_j),\quad x\in\R,\end{equation} which satisfies
$$\I_\phi f(x_k) = f(x_k),\quad k\in\Z.$$
If the reliance on the sequence $X$ needs to be clear, the notation $\I_\phi^X f$ will be used.

\begin{comment}Such schemes have been considered in the nonuniform sampling context in \cite{HammJAT,HammZonotopes,LedfordScattered,schsiv}, while similar concerns in the uniform setting have been studied in \cite{HammLedford,HammLedford2,LedfordCardinal,Ledford_ellp}.  A survey of the results obtained for the interpolation scheme involving the so-called general multiquadrics, and the basis for the work done in the sequel may be found in \cite{HammSampTA}, though the necessary background will be restated here for completeness.\end{comment}

As this article continues the theme of the SampTA 2015 article by the second author \cite{HammSampTA}, the sequel will primarily emphasize interpolation using the so-called {\em general multiquadrics} as kernels, though we stress that the results here are part of a more general phenomenon and have analogues for a variety of different kernels, which we briefly mention in Section~\ref{SECExtensions}.  The general multiquadrics are defined using two parameters via $$\phi_{\alpha,c}(x):=(|x|^2+c^2)^\alpha.$$  To avoid notational encumbrance, we adopt the convention $\I_{\alpha,c}:=\I_{\phi_{\alpha,c}}$.  %It should be noted that all of the techniques and results in what follows have analogues for many different kernels; however, for ease of exposition, we focus on the general multiquadrics and note the extensions in Section \ref{SECExtensions}.

Let us first note that if $c>0$, $\alpha<-1/2$ and $X$ is a CIS for $PW_\pi$, then for any $f\in PW_\pi$, a multiquadric interpolant $\I_{\alpha,c}f$ as in \eqref{EQInterpForm} exists \cite{LedfordScattered}.  Furthermore, the interpolant is unique (i.e. the sequence $(a_j)$ in \eqref{EQInterpForm} is uniquely determined as the solution to the equation $\mathcal{M}a=y$, where $y_j=f(x_j)$, and $\mathcal{M}:=(\phi(x_j-x_k))_{j,k\in\Z}$), and $\I_{\alpha,c}$ is a bounded linear operator from $PW_\pi\to C_0\cap L_2(\R)$ \cite[Proposition 1]{LedfordScattered}.

%%%%%%%%%%%%%%%%%%%%%%%%%%%%%%%%%%%%%%%%%%%%%%%%%
%%%%%%  Recovery Results                    %%%%%
%%%%%%%%%%%%%%%%%%%%%%%%%%%%%%%%%%%%%%%%%%%%%%%%%

\subsection{Bandlimited Recovery}\label{SECRecovery}

Here we recall one of the main recovery results for bandlimited functions using the interpolation method set out in the previous section.  

\begin{theorem}[cf. \cite{HammSampTA}\label{THM1DRecovery}, Theorem IV.1]\label{THMScatteredRecovery}
Let $\alpha<-1/2$ and let $X$ be a complete interpolating sequence for $PW_\pi$.  If $f\in PW_\pi$, then $\I_{\alpha,c}^Xf\in L_2(\R)$ and
$$\inflim{c}\|\I_{\alpha,c}^Xf-f\|_{L_2(\R)} = 0,$$
and
$$\inflim{c}|\I_{\alpha,c}^Xf(x)-f(x)| = 0$$
uniformly on $\R$.

Moreover, if $f\in PW_\sigma$ for some $\sigma<\pi$,
\begin{equation}\label{EQOversamplingRate}
\|\I_{\alpha,c}^Xf-f\|_{L_2(\R)}\leq C e^{-c(\pi-\sigma)}\|f\|_{L_2(\R)},
\end{equation}
where the constant $C$ depends on $\alpha$ and $X$, but not on $c$.
\end{theorem}

While the first part of this theorem only says something about the asymptotic behaviour of the interpolants for functions whose Fourier transform is fully supported in the band of the Paley-Wiener space, we nonetheless obtain exponential convergence in terms of the shape parameter, $c$, of the multiquadric when oversampling, corresponding to the same notions in classical sampling theory.  Currently, no approximation rates in terms of $c$ are known when $\widehat{f}$ has support on the full interval $[-\pi,\pi]$.

For a multivariate analogue of Theorem \ref{THMScatteredRecovery}, see \cite[Theorem IV.2]{HammSampTA}; since its statement is somewhat technical and would take us out of the scope considered here, we omit it.

\begin{comment}The following shows similar convergence phenomena in higher dimensions.  We eliminate some of the details, but the main idea is that since it is unknown whether or not Riesz bases of exponentials for the Euclidean ball exist, one approximates the ball with convex bodies (zonotopes) which do have such bases in certain cases.  Additionally, the interpolation is of functions in the Paley-Wiener space over a smaller ball, and thus exponential decay rates are achieved because the method is one of oversampling.

\begin{theorem}[\cite{HammZonotopes}, Theorem 4.7]\label{THMMultivariaterecovery}
Let $\alpha<-d/2$.  Suppose that $\delta\in(2/3,1)$ and $\beta\in(0,3\delta-2)$.  Suppose that $S$ is a symmetric convex body such that $\delta B_2^d\subset S\subset B_2^d$ and that $(e^{- i\bracket{x_k,\cdot}})_{k\in\Z}$ is a Riesz basis for $L_2(S)$.  Then for every $f\in PW_{\beta B_2}$, 
$$\inflim{c}\|\I_{\alpha,c}^Xf-f\|_{L_2(\R^d)}=0,$$
and
$$\inflim{c}|\I_{\alpha,c}^Xf(x)-f(x)|=0$$
uniformly on $\R^d$.
Moreover, there exists a constant $C>0$ such that for every $f\in PW_{\beta B_2}$, 
\begin{equation}\label{EQOversamplingRateMultivariate}
\max\{\|\I_{\alpha,c}^Xf-f\|_{L_2(\R^d)},\|\I_{\alpha,c}^Xf-f\|_{L_\infty(\R^d)}\}\leq C\left(\frac{c\beta}{\delta}\right)^{\alpha+\frac{d+1}{2}}e^{-c(3\delta-2-\beta)}\|f\|_{L_2(\R^d)},
\end{equation}
where $C$ is independent of $c$.
\end{theorem}
\end{comment}

%%%%%%%%%%%%%%%%%%%%%%%%%%%%%%%%%%%%%%%%%%%%%%%%%
%%%%%%  Noise                               %%%%%
%%%%%%%%%%%%%%%%%%%%%%%%%%%%%%%%%%%%%%%%%%%%%%%%%

\section{Interpolation in the Presence of Noise}\label{SECNoise}

Given the preliminaries above, we turn our attention to considering how the interpolation scheme behaves in the presence of noise.  There are two main kinds of noise that will be considered: noisy data, and so-called {\em jitter error}.

\subsection{Stability under perturbation of sample points}

Jitter error corresponds to the case when the sample points $X$ are perturbed.  That is, instead of sampling at $X:=(x_j)_{j\in\Z}$, we sample at $\tilde{X}:=(\tilde{x_j})_{j\in\Z}$ with $\tilde{x_j}=x_j+\eps_j$ for some bounded perturbation $(\eps_j)\in \ell_\infty(\Z)$.  Physically, this may correspond to non-ideal sensors which have some error in the timing of the sampling.

Notice that it follows from Theorem \ref{THMScatteredRecovery} that the recovery results therein are independent under perturbations of the sample points at least as long as the perturbed points still form a complete interpolating sequence.  So if $\tilde{X}=X+\eps$ is a complete interpolating sequence for the Paley--Wiener space, we have
$$\inflim{c}\|I_{\alpha,c}^{\tilde{X}}f-f\|_{L_2} =\inflim{c}\|I_{\alpha,c}^Xf-f\|_{L_2} = 0.$$
Of course, the rate of convergence may differ, though it is difficult to relate how. As a particular example of this, the following proposition is evident.

\begin{proposition}\label{PROPJitterKadecs}
Suppose that $X$ satisfies Kadec's 1/4--Theorem, and $\underset{j\in\Z}\sup|x_j-j|= L<1/4$. Then if $\|(\eps_j)_j\|_\infty <  1/4-L$, $\tilde{X}$ given by $\tilde{x_j}=x_j+\eps_j$ is a CIS for $PW_\pi$.  In particular, $\inflim{c}\I_{\alpha,c}^{\tilde{X}}f=f$ in $L_2$ and uniformly on $\R$ for any $f\in PW_\pi$.
\end{proposition}

\begin{proof}
Notice that $\tilde{X}$ still satisfies the condition of Kadec's Theorem and apply Theorem \ref{THMScatteredRecovery}.
\end{proof}
 
Similarly, if $\eps_j=0$ for all but finitely many $j$, Theorem \ref{THMPerturbedCIS}(ii) implies that $\tilde{X}$ is again a CIS.  Theorem \ref{THMPerturbedCIS}(i) also implies that for any CIS $X$, there exists a constant $L$ such that if $\|\eps\|_{\ell_\infty}\leq L$, then $\tilde{X}$ is again a CIS, and consequently the conclusion of Proposition \ref{PROPJitterKadecs} holds in both of these cases.

However, one drawback is that this $L$ may be very small.  One can see this, for example, because the 1/4-Theorem is sharp, so if $X$ was perturbed from the integer lattice arbitrarily close to 1/4, a small perturbation might fail.  %Using the example in Theorem \ref{THMKadecs}, one can take $x_n=n-1/4-\eps$ for $n>0$ and $x_n=-n+1/4+\eps$ for $n<0$, where $\eps>0$ is fixed. Then there is a perturbation of $X$ of norm $\eps$ which fails to be a CIS.

Nonetheless, we may make some estimate on $L$ based not on the magnitude of $\|\eps\|_{\ell_\infty}$, but on the so-called {\em frame bounds} of the basis $(e^{-ix_j(\cdot)})_j$.  Note that it follows from \eqref{EQFrameBounds} and Plancherel's Identity that there are constants $0<A\leq B<\infty$ (the frame bounds), such that for any $f\in PW_\pi$,
\begin{equation}\label{EQframe}
A\|f\|_{L_2(\R)}^2\leq\zsum{j}|f(x_j)|^2\leq B\|f\|_{L_2(\R)}^2.
\end{equation}
The following can be found in \cite{DuffinSchaeffer}:
\begin{proposition}\label{PROPPerturbedFrameBounds}
Suppose that $(e^{-ix_j(\cdot)})_{j\in\Z}$ is a Riesz basis for $L_2[-\pi,\pi]$ with frame bounds $A,B>0$. Then if $0<L<\pi^{-1}\ln\left(\sqrt{\frac{A}{B}}+1\right)$, and $\tilde{x_j}=x_j+\eps_j$, with $\|(\eps_j)_j\|_{\ell_\infty}\leq L$, $(e^{-i\tilde{x_j}(\cdot)})_{j\in\Z}$ is a Riesz basis for $L_2[-\pi,\pi]$ with frame bounds $A(1-\sqrt{C})^2$ and $B(1+\sqrt{C})^2$, where $C=\frac{B}{A}(e^{\pi L}-1)^2$.
\end{proposition}

Thus combining Proposition \ref{PROPPerturbedFrameBounds} and Theorem \ref{THMScatteredRecovery} gives a bound on the magnitude of jitter error allowed for a given CIS.

\subsection{Robustness to noisy samples}

Next, consider what happens if, instead of sampling $f(x_j)$ exactly, we actually measure $y_j = f(x_j)+\delta_j$.  For now, assume that $(\delta_j)\in\ell_2$, and $\|(\delta_j)_j\|_{\ell_2}\leq\delta$.  In this case, the noise is added to the signal, and could appear as random background noise, or in some cases deterministic (possibly adversarial) noise.  There are many ways to model such noise, but our focus here will be on that which is square-summable.

Given noisy samples $(y_j)$ as above, let $\widetilde{\I_{\alpha,c}^X}f$ be the interpolant of the data $y_j$ satisfying $\widetilde{\I_{\alpha,c}^X}f(x_k)=y_k$, $k\in\Z$.  Note that $(y_j)\in\ell_2$ by the condition on the noise sequence $(\delta_j)$.  Consequently, on account of Definition \ref{DEFCIS}, there is a unique $g\in PW_\pi$ such that $g(x_j)=y_j$.  Thus, by uniqueness of the interpolant, there is a unique $g\in PW_\pi$ such that  $\widetilde{\I_{\alpha,c}^X}f=\I_{\alpha,c}^Xg$, and the following holds.

\begin{theorem}\label{THMnoise}
Let $X$ be a CIS for $PW_\pi$ and let $y_j=f(x_j)+\delta_j$ for some $f\in PW_\pi$ with $\|(\delta_j)\|_{\ell_2}\leq\delta$.  Then
$$\|\widetilde{\I_{\alpha,c}^X}f-f\|_{L_2}\leq \dfrac{\delta}{\sqrt{A}} + o(1),\quad c\to\infty,$$ where $A$ is as in \eqref{EQframe}.  Moreover, if $f\in PW_\sigma$, $\sigma<\pi$, the $o(1)$ term may be taken to be a constant multiple of $e^{-c(\pi-\sigma)}\|f\|_{L_2}$.
\end{theorem}

\begin{proof}
Let $g\in PW_\pi$ be the function described above. Then we have
$$\|\widetilde{\I_{\alpha,c}^X}f-f\|_{L_2} = \|\I_{\alpha,c}^Xg-f\|_{L_2}\leq \|\I_{\alpha,c}^Xg-g\|_{L_2}+\|g-f\|_{L_2} =: N_1+N_2.$$

It follows from Theorem \ref{THMScatteredRecovery} that $N_1 = o(1),\; c\to\infty$.  Applying \eqref{EQframe}, we estimate $N_2$ as follows:
$$\|g-f\|_{L_2}\leq A^{-\frac{1}{2}}\|(g(x_j)-f(x_j))_j\|_{\ell_2} = A^{-\frac{1}{2}}\|(\delta_j)_j\|_{\ell_2}\leq A^{-\frac{1}{2}}\delta,$$ which concludes the proof of the first statement.  The moreover statement follows directly on account of Theorem \ref{THMScatteredRecovery}.
\end{proof}

%%%%%%%%%%%%%%%%%%%%%%%%%%%%%%%%%%%%%%%%%%%%%%%%%
%%%%%%  Approximation Rates                 %%%%%
%%%%%%%%%%%%%%%%%%%%%%%%%%%%%%%%%%%%%%%%%%%%%%%%%

\section{Approximation Rates Based on Spacing}\label{SECRates}

As discussed in the previous section, the approximation rates in terms of the shape parameter, $c$, of the multiquadric are maintained in the presence of noise (hence the error of approximation is dominated by the $\ell_2$ norm of the noise as in Theorem \ref{THMnoise}).  But another type of approximation rate has been considered.  Specifically, we fix a CIS, $X$, and consider interpolation at $hX$ for $0<h\leq1$, and we tune the shape parameter of the multiquadric to reflect the dilation.  More precisely, we interpolate from the space 
$$\left\{\zsum{j}a_j\phi_{\alpha,1}(\cdot-hx_j) = \zsum{j}a_j\left(|\cdot-hx_j|^2+1\right)^\alpha: (a_j)_{j\in\Z}\subset\R\right\},$$
which may be closed in, say, the topology of uniform convergence on compact sets. 
In the uniform interpolation setting, when $X=\Z$, these interpolants have a special structure, which will be discussed later. To emphasize the distinction (and the reliance of the shape parameter on $h$) we write this new interpolant in a different manner as $I_\alpha^{hX}f$.  One may show that the relation to the original interpolant is 
\begin{equation}\label{EQIRelation}
I_\alpha^{hX}f(x) = \frac{1}{h}\I_{\alpha,h^{-1}}f^h\left(\frac{x}{h}\right),
\end{equation}
where $f^h(x):=hf(hx)$.  When $\alpha$ is fixed, we drop the subscript to ease the notation.

To begin the analysis of the effect of noise on this process, let us recall the following.

\begin{theorem}[\cite{HammJAT}, Theorem 3.4]\label{THMmeshApproximationRate}
Suppose that $\alpha<-1/2$, $k\in\N$, $0<h\leq1$, and $X$ is a CIS for $PW_\pi$.  Then there exists a constant $C$ independent of $h$ such that for every $f\in W_2^k(\R)$,
\begin{equation}\label{EQhrate}
\|I^{hX}f-f\|_{L_2(\R)}\leq Ch^k|f|_{W_2^k(\R)}.
\end{equation}
\end{theorem}

\begin{remark} Note again that the estimate in Theorem \ref{THMmeshApproximationRate} is invariant under perturbing the CIS in a certain manner.  Specifically, if $X$ is a CIS for $PW_\pi$, and so is $Y$, then \eqref{EQhrate} holds for both interpolants albeit with a different constant $C$.  Moreover, one finds via the triangle inequality that
\begin{equation}\label{EQDifferenceBound}
\|I^{hX}f-I^{hY}f\|_{L_2}\leq (C_X+C_Y)h^k|f|_{W_2^k}.
\end{equation}
\end{remark}

\subsection{Noise in nonuniform interpolation}

To discuss the question of reconstruction from noisy data in the setting described in this section, it is pertinent to recall a theorem on the stability of interpolating a given Sobolev function via a bandlimited one as an intermediate step to analyzing the interpolant.  

\begin{theorem}[\cite{HammJAT}, Theorem 3.1]\label{THMinterpSbyB}
Let $k\in\N$, $h>0$, and let $X$ be a fixed CIS for $PW_\pi$.  There exists a constant $C=C_{k,X}$, independent of $h$, such that for every $f\in W_2^k(\R)$, there exists a unique $F\in PW_\frac{\pi}{h}$ such that
$$F(hx_j) = f(hx_j),\quad j\in\Z,$$
$$|F|_{W_2^k}\leq C|f|_{W_2^k},$$
and
$$\|F-f\|_{L_2}\leq Ch^k|f|_{W_2^k}.$$
\end{theorem}

Additionally, the interpolation operators are uniformly bounded in the Sobolev seminorm:

\begin{theorem}[\cite{HammJAT}, Theorem 3.3]\label{THMuniformhbound}
For each $k\geq0$, there is a constant $C$, independent of $h$, such that
$$|I^{hX}f|_{W_2^k}\leq C|f|_{W_2^k},\quad f\in W_2^k(\R).$$
\end{theorem}

Now, suppose that we measure $y_j(h) = f(hx_j)+\delta_j(h)$ with $\underset{h}\sup\|(\delta_j(h))_j\|_{\ell_2}\leq\delta$,  and let $\widetilde{I^{hX}}f$ be the interpolant of $y_j(h)$. Then we have the following rate of approximation.

\begin{theorem}
Under the assumptions of Theorem \ref{THMinterpSbyB}, there is a constant $C=C_{k,X}$ such that, for every $f\in W_2^k(\R)$,
$$\|\widetilde{I^{hX}}f-f\|_{L_2}\leq Ch^k|f|_{W_2^k}+\dfrac{\delta}{\sqrt{A}},$$
where $A$ is the lower frame bound for $X$ given by \eqref{EQframe}.
\end{theorem}

\begin{proof}
The first key observation is that $hX$ is a CIS for $PW_\frac{\pi}{h}$, which can be seen because Riesz bases are preserved under bounded, invertible linear transformations.  Consequently, let $g\in PW_\frac{\pi}{h}$ be the unique function such that $g(hx_j)=y_j(h)$.  Then we have
$$\|\widetilde{I^{hX}}f-f\|_{L_2} = \|I^{hX}g-f\|_{L_2}\leq\|I^{hX}g-I^{hX}f\|_{L_2}+\|I^{hX}f-f\|_{L_2}=:N_1+N_2.$$

Theorem \ref{THMmeshApproximationRate} implies that $N_2\leq Ch^k|f|_{W_2^k}$.  Secondly, Theorem \ref{THMuniformhbound} with $k=0$ implies that
$$N_1\leq\|I^{hX}(g-f)\|_{L_2}\leq C\|g-f\|_{L_2}.$$
Let $F\in PW_\frac{\pi}{h}$ be the function given by Theorem \ref{THMinterpSbyB} such that $F(hx_j)=f(hx_j)$.  Then we have
$$\|g-f\|_2\leq\|g-F\|_2+\|F-f\|_2\leq\|g-F\|_2+Ch^k|f|_{W_2^k}.$$
From \eqref{EQframe}, $$\|g-F\|_2\leq A^{-\frac{1}{2}}\|(g(hx_j)-f(hx_j))_j\|_{\ell_2}=A^{-\frac{1}{2}}\|(\delta_j(h))_j\|_{\ell_2}\leq A^{-\frac{1}{2}}\delta.$$ Thus $N_1\leq Ch^k|f|_{W_2^k}+A^{-\frac{1}{2}}\delta$.  

Combining the estimates on $N_1$ and $N_2$ completes the proof.
\end{proof}

\subsection{Uniform Sampling}

In the uniform case (when $X$ is a lattice), this interpolation scheme has some special properties, including the possibility of using growing kernels such as multiquadrics with positive $\alpha$.  Additionally, the interpolants themselves have a simpler form as they lie in the span of shifts of a single {\em cardinal function} which behaves like the classical cardinal sine function.

Given a multiquadric, we form a cardinal function $L_{\alpha,c}$ satisfying $L_{\alpha,c}(j)=\delta_{0,j},\;j\in\Z$, which lies in the closure of the linear span of $(\phi_{\alpha,c}(\cdot-j))_{j\in\Z}$ (where convergence is uniform for sufficiently negative $\alpha$, and uniform on compact subsets for positive $\alpha$.  Then the multiquadric interpolant can be written as 
\begin{equation}\label{EQUniformInterpDef}\I_{\alpha,c}^{\Z}f(x):=\zsum{j}f(j)L_{\alpha,c}(x-j).\end{equation}

The cardinal function $L_{\alpha,c}$ can be defined by its Fourier transform:
\begin{equation}\label{EQFundamentalFT}
\Lachat(\xi):=\dfrac{\phica(\xi)}{\zsum{k}\phica(\xi+2\pi k)},\quad \xi\in\R\setminus\{0\}.
\end{equation}

Note that the series in \eqref{EQUniformInterpDef} is rather similar to the series in the classical Whittaker-Kotelnikov-Shannon sampling theorem if $f\in PW_\pi$, but with the sinc function replaced by the cardinal function associated with the general multiquadric.  This indeed was the observation made by Schoenberg, who instigated the study of cardinal spline interpolation.  The point was to apply a sort of summability method to the sinc series in the sampling theorem because the fact that $\sinc(x)=O(|x|^{-1})$ implies that it takes a rather large number of terms to well-approximate the series via truncation.  If the cardinal functions $L_{\alpha,c}$ decay faster than the sinc function, then the series in \eqref{EQUniformInterpDef} will be easier to approximate by truncation; of course, the question then is whether or not the cardinal interpolant $\I_{\alpha,c}^\Z f$ is close to $f$ (either in $L_2$ or uniformly depending on the kind of guarantees one desires).  For an in-depth study of the decay rates of the cardinal functions associated with general multiquadrics, see \cite{HammLedford}, especially Section 4 therein.  For a broad range of values of $\alpha$, $L_{\alpha,c}$ decays faster than $|x|^{-1}$.  We note also that for the cases $\alpha=\pm1/2$, such considerations were already made by Buhmann \cite{Buhmann} and Riemenschneider and Sivakumar \cite{RS1}.  Additionally, $L_p$ approximation rates of the same order as in Theorem \ref{THMmeshApproximationRate} for interpolation at $h\Z$ can be found in \cite{HammLedford2}.

\subsection{Interpolation of Compactly Supported Functions}

Let us now turn our attentions to some more practical considerations which may prove useful in applications.  Of course, everything from Theorem \ref{THMnoise} to Theorem \ref{THMmeshApproximationRate} holds whenever we take $X=\Z$, which is clearly a CIS for $PW_\pi$.  Let us consider for the moment what happens whenever we consider interpolation of univariate compactly supported Sobolev functions.  To wit, let $f\in W_2^k(\R)$ with $\supp(f)\subset[-1,1]$ (this choice of interval is arbitrary for ease of presentation, and can easily be dilated).  Then for $N\in\N$, the interpolant $I^{N^{-1}\Z}f$ is actually interpolating $f$ at the sequence $\{\frac{j}{N}\}_{j=-N}^N$, and has the following form by combining \eqref{EQIRelation} and \eqref{EQUniformInterpDef}:
\begin{equation}\label{EQcardinaldef}
I^{N^{-1}\Z}f(x) = \finsum{j}{-N}{N}f\left(\frac{j}{N}\right)L_{\alpha,c}(Nx-j).
\end{equation}

Consequently, Theorem \ref{THMmeshApproximationRate} shows that the approximation rate in this case has an upper bound in terms of the number of sample points.  Namely, if $X = \{\frac{j}{N}\}_{j=-N}^N$, then
\begin{equation}
 \|I^{N^{-1}\Z}f-f\|_{L_2}\leq CN^{-k}|f|_{W_2^k} = C|X|^{-k}|f|_{W_2^k}.
\end{equation}
Also, Theorem \ref{THMnoise} still holds with $h$ replaced by $|X|^{-1}$ as well.

But now, consider $N^{-1}X=\{\frac{x_j}{N}\}_{J=-N}^N$ to be an arbitrary set of distinct points in $[-1,1]$ (i.e. $x_j$ are arbitrary in $[-N,N]$).  Then let 
\begin{displaymath}
 \widetilde{x_k}:=\left\{ \begin{array}{ll}
                          k,  & |k|>N\\
                          x_k,  & |k|\leq N.\\
                         \end{array}\right.
\end{displaymath}

Evidently, $\widetilde{X}$ is a CIS for $PW_\pi$ on account of Theorem \ref{THMPerturbedCIS}(ii).  It follows from \eqref{EQDifferenceBound} that
\begin{align}\label{EQscatteredapprox}
\|I^{N^{-1}\widetilde{X}}f-f\|_{L_2} & \leq\|I^{N^{-1}\widetilde{X}}f-I^{N^{-1}\Z}f\|_{L_2}+\|I^{N^{-1}\Z}f-f\|_{L_2}\nonumber\\
& \leq(C_{\widetilde{X}}+2C_\Z)N^{-k}|f|_{W_2^k}.
\end{align}
%Again, the right hand side may be written as a constant multiple of $(C_X+2C_\Z)|X|^{-k}|f|_{W_2^k}$.

The constant $C_{\widetilde{X}}$ depends on a few things: the minimum spacing of the sequence (i.e. $\inf_{j\neq k}|x_j-x_k|$), and the frame bounds for the basis as in \eqref{EQframe}.  Since the lower frame bound $A$ can degenerate as the minimal distance between the points shrinks, we cannot hope to have a uniform constant for all $\widetilde{X}$.  However, for a class of small perturbations of the integers, the constant can be uniform.  Indeed, combining \eqref{EQscatteredapprox} with Proposition \ref{PROPPerturbedFrameBounds} yields the following.

\begin{theorem}\label{THM6.6}
Let $L<\ln(2)/\pi$.  There is a constant $C$ such that for any $\widetilde{X}$ with $\sup_{k\in\Z}|\widetilde{x_k}-k|<L$, 
$$\|I^{N^{-1}\widetilde{X}}f-f\|_{L_2}\leq CN^{-k}|f|_{W_2^k},$$ for every $f\in W_2^k(\R)$ which is supported on a closed interval.
\end{theorem}

This theorem and the estimate in \eqref{EQscatteredapprox} imply that if one wishes to approximate a compactly supported $f$ by its interpolant $I^{N^{-1}\widetilde{X}}f$, it suffices to consider the more simple uniform interpolant $I^{N^{-1}\Z}f$ up to the penalty of a possibly larger constant $C$.  The usefulness of this will be discussed further in the next section.

%%%%%%%%%%%%%%%%%%%%%%%%%%%%%%%%%%%%%%%%%%%%%%%%%%%%%%%%
%%%%%%%%%%%%%%%%%%%%%%%%%%%%%%%%%%%%%%%%%%%%%%%%%%%%%%%%
\section{Computational Feasibility}\label{SECComputations}
%%%%%%%%%%%%%%%%%%%%%%%%%%%%%%%%%%%%%%%%%%%%%%%%%%%%%%%%
%%%%%%%%%%%%%%%%%%%%%%%%%%%%%%%%%%%%%%%%%%%%%%%%%%%%%%%%

In this section, we investigate the computational feasibility, peculiarities, and potential advantages of the interpolation method based on cardinal functions compared to traditional RBF theory.
Again, for ease of presentation we limit our discussion to problems in one dimension. 
As discussed in Section \ref{SECRates}, consider a function $f\in W_2^k(\R)$ whose support lies inside $[-1,1]$, and we interpolate at a sequence of points $(x_j)_{j=-N}^N\subset[-1,1]$.  The classical method using RBFs is to do interpolate from the linear span of $\{\phi_{\alpha,c}(\cdot-x_j):j=-N,\cdots,N\}$.  The drawback in this case if $\phi$ is a multiquadric or the Gaussian kernel is that forming the interpolant can be computationally quite expensive as it is formed by inverting the matrix $\mathcal{M}_N:=(\phi_{\alpha,c}(x_k-x_j))_{k,j=-N}^N$ and applying it to the data vector $y_j = f(x_j)$ to determine the coefficients of the interpolant. Part of the problem is that if the minimum spacing of the points is $h$, then the condition number for this matrix can be as bad as $e^{1/h^2}$ (\cite{NarcSivaWard}), which is undesirable. The other disadvantage of this framework is that it is not robust to noise. 
RBF interpolation is very good at recovering smooth functions, but is sensitive to noise unless other smoothing techniques are applied.

On the other hand, given a sequence $(x_j)_{j=-N}^N\subset[-1,1]$, by \eqref{EQscatteredapprox} and Theorem \ref{THM6.6}, we may simply use the uniform interpolant $I^{N^{-1}\Z}f$ to approximate $f$. 
The benefit of this, is that $I^{N^{-1}\Z}f$ is less difficult to compute. 
Indeed, one must first estimate the function $\widehat{L_{\alpha,c}}$  by truncating the series in the denominator of \eqref{EQFundamentalFT},  then evaluate $L_{\alpha,c}$ via the Fast Fourier Transform (FFT). 
Then one directly forms the series in \eqref{EQcardinaldef} from the already known sample values $f(j/N)$.  Moreover, as discussed in both of the previous two sections, this method enjoys the advantage of being robust to noise.  Notice however, that this interpolation scheme is different in the sense that $I^{N^{-1}\Z}f$ is in the span of $(L_{\alpha,c}(\cdot-j))_{j=-N}^N$, which in turn is in the span of $(\phi_{\alpha,c}(\cdot-j))_{j\in\Z}$, as opposed to only the span of $2N+1$ translates of the multiquadric.

\subsection{Approximation of the Fourier transform of the cardinal function}
As stated above, one first needs to truncate the series in the denominator of the Fourier transform of the cardinal function. It is known (see for instance~\cite[Theorem 8.15]{Wendland}) that the Fourier transform of a multiquadric is given (in one dimension) by 
\begin{equation}
\label{eq:FourierMultiquadric}
\widehat{\phi_{\alpha,c}}(\xi) = \sqrt{2\pi}\frac{2^{1+\alpha}}{\Gamma(-\alpha)}\left( \frac{c}{|\xi|} \right)^{\alpha+1/2}K_{\alpha+1/2}(c|\xi|), \quad \xi \in \R \backslash \{0 \},
\end{equation}
where $K_\nu(r) := \int_{0}^{\infty} e^{-r\cosh(t)}\cosh(\nu t) \textrm{d}t, \quad r > 0, \nu \in \R$ is the modified Bessel function of the second kind. Note that these Bessel functions have a pole at the origin and decay exponentially. 

It follows that the truncation of the series in the Fourier transform of the cardinal function associated with the general multiquadrics is possible thanks to the fast decay of the Bessel function.  In particular, we have the following.

\begin{theorem}
\label{thm:truncation}
Let $\varepsilon > 0$. Let $\alpha \in \R$ and $\alpha < 0$. For any $c >0$, there exists a natural number $\tau := \tau_{c,\alpha,\varepsilon} \in \N$, such that for all $\xi \in \R$, there exists a $k_\xi \in \Z$ with 
\begin{equation}
\label{eq:truncation}
\left| \zsum{k} \widehat{\phi_{\alpha,c}}(\xi + 2\pi k) - \sum_{k = k_\xi - \tau}^{k_\xi + \tau} \widehat{\phi_{\alpha,c}}(\xi + 2\pi k) \right| \leq \varepsilon \left| \zsum{k} \widehat{\phi_{\alpha,c}}(\xi + 2\pi k)\right|.
\end{equation} 
\end{theorem}

Before we prove this result, let us remark that \eqref{eq:truncation} is sufficient for a precision within $\varepsilon$ of the truncated version of the Fourier transform of the cardinal function. 
To see this, let us denote $S_{\alpha,c}(\xi) := \zsum{k} \widehat{\phi_{\alpha,c}}(\xi + 2\pi k)$ and $S_{\alpha,c}^{\tau}(\xi) := \sum_{-\tau \leq k \leq \tau} \widehat{\phi_{\alpha,c}}(\xi + 2\pi k)$. 
Then $\left|\frac{\widehat{\phi}(\xi)}{S_{\alpha,c}(\xi)}- \frac{\widehat{\phi}(\xi)}{S_{\alpha,c}^{\tau}(\xi)}\right|  = \dfrac{|S_{\alpha,c}^{\tau}(\xi)-S_{\alpha,c}(\xi)|}{|S_{\alpha,c}(\xi)|\cdot |S_{\alpha,c}^{\tau}(\xi)|}|\widehat{\phi}(\xi)|$.
Together with \eqref{eq:truncation}, it follows $\left|\frac{\widehat{\phi}(\xi)}{S_{\alpha,c}(\xi)}- \frac{\widehat{\phi}(\xi)}{S_{\alpha,c}^{\tau}(\xi)}\right| \leq \varepsilon \frac{\widehat{\phi}(\xi)}{S_{\alpha,c}^{\tau}(\xi)}$. 
All terms in $S_{\alpha,c}^{\tau}(\xi)$ are positive and since $\widehat{\phi}(\xi)$ is included as the case $k = 0$, it follows that $\left|\frac{\widehat{\phi}(\xi)}{S_{\alpha,c}(\xi)}- \frac{\widehat{\phi}(\xi)}{S_{\alpha,c}^{\tau}(\xi)}\right| \leq \varepsilon$.

\begin{proof}
First note that $S_{\alpha,c}(\xi)$ (defined above) is $2\pi$-periodic.
It is straightforward to see that \eqref{eq:truncation} is equivalent to finding $\tau$ such that for any $\xi^* \in (-\pi,\pi)$
$\left| \zsum{k} \widehat{\phi_{\alpha,c}}(\xi^* + 2\pi k) - \dsum_{k = - \tau}^{\tau} \widehat{\phi_{\alpha,c}}(\xi^* + 2\pi k) \right| \leq \varepsilon \left| \zsum{k} \widehat{\phi_{\alpha,c}}(\xi^* + 2\pi k)\right|$.

From~\cite[Lemma 5.13]{Wendland} it follows that for $\nu \in \R$, $0\leq K_\nu(r) \leq \sqrt{2\pi}r^{-1/2}e^{-r}e^{\nu^2/(2r)}$. 
In particular, let $r_k := \xi^* + 2\pi k > 0$, for some $k > \tau$; it follows, with $\nu = \alpha + 1/2$ that
\begin{equation}
\widehat{\phi_{\alpha,c}}(r_k) \leq \frac{2^{1+\alpha}}{\Gamma(-\alpha)}c^{\alpha}2\pi r_k^{-\alpha-1}e^{-cr_k}e^{\frac{(\alpha+1/2)^2}{2cr_k}}.\nonumber
\end{equation}
With $\lambda := \frac{2^{1+\alpha}}{\Gamma(-\alpha)}c^{\alpha}2\pi,$ this expression simplifies to 
$\widehat{\phi_{\alpha,c}}(r_k) \leq \lambda r_k^{-\alpha-1}e^{-cr_k}e^{\frac{(\alpha+1/2)^2}{2cr_k}}$. 
For $k$ large enough, there exists a constant $\gamma > 0$ such that 
\begin{equation}
\label{eq:expDecayPhi}
\widehat{\phi_{\alpha,c}}(r_k) \leq \gamma e^{-cr_k}.
\end{equation}
Plugging back in the definition of $r_k$ yields 
\begin{equation}
\widehat{\phi_{\alpha,c}}(r_k) \leq \gamma e^{-c\xi^*}e^{-2\pi c k}, \quad \text{ for any } k > \tau.\nonumber
\end{equation}
Similarly, given $k < -\tau$, we arrive at the following estimate:
\begin{equation}
\widehat{\phi_{\alpha,c}}(r_k) \leq \gamma e^{c\xi^*}e^{2\pi c k}, \quad \text{ for any } k < -\tau.\nonumber
\end{equation}
Summing for all $k$ outside of $\{-\tau, \cdots, \tau\}$ finally yields, for $\xi \in \R$,
\begin{align}
\bigg| \zsum{k} \widehat{\phi_{\alpha,c}}(\xi + 2\pi k) &- \sum_{k = k_\xi - \tau}^{k_\xi + \tau} \widehat{\phi_{\alpha,c}}(\xi + 2\pi k) \bigg| = \sum_{k < -\tau} \widehat{\phi_{\alpha,c}}(\xi^* + 2\pi k) + \sum_{k > \tau} \widehat{\phi_{\alpha,c}}(\xi^* + 2\pi k) \nonumber\\
&\leq \sum_{k < -\tau} \gamma e^{c\xi^*}e^{2\pi c k} + \sum_{k > \tau} \gamma e^{-c\xi^*}e^{-2\pi c k} = \left(e^{c\xi^*} + e^{-c\xi^*}\right)\gamma \frac{e^{-2\pi c(\tau+1)}}{1-e^{-2\pi c}} \nonumber\\ 
&\leq \frac{2\gamma \operatorname{cosh}(c\xi^*)}{1-e^{-2\pi c}}e^{-2\pi c (\tau+1)}.\nonumber
\end{align}

The sum for $k \in \Z$ can be estimated from below (see for instance~\cite[Proof of Prop.3.2]{HammLedford}) by 
\begin{equation}
\zsum{k} \widehat{\phi_{\alpha,c}}(\xi) \geq De^{-4\pi c},\nonumber
\end{equation}
for a certain constant $D := D_{\alpha,c} > 0$. 
Therefore, for \eqref{eq:truncation} to be valid, it suffices to find $\tau$ such that 
\begin{equation}
\dfrac{2\gamma\operatorname{cosh}(c\xi^*)}{1-e^{-2\pi c}}e^{-2\pi c(\tau+1)}\leq\varepsilon De^{-4\pi c},\nonumber
%\left(e^{c\xi^*} + e^{-c\xi^*}\right)\gamma \frac{e^{-2\pi c(\tau+1)}}{1-e^{-2\pi c}} \leq \varepsilon D e^{-4\pi c},\nonumber
\end{equation}
which is achieved whenever 
\begin{equation}
\tau \geq 1 + \frac{\ln(\varepsilon^{-1})}{2\pi c} + \frac{\ln\left( 2\gamma \operatorname{cosh}(c\pi) \right)}{2\pi c} - \frac{\ln\left( D(1-e^{-2\pi c}) \right)}{2\pi c}.
\end{equation}
\end{proof}

\begin{remark}\label{Remark:D}
A careful analysis of the proof of Proposition 3.2 from~\cite{HammLedford} gives insight on how to pick $D$.
For instance, for $0 > \alpha \geq -1$, one can choose $D \leq \beta\frac{2^{1+\alpha}}{\Gamma(-\alpha)}c^{\alpha}(2\pi)^{-\alpha-1}e^{-2\pi c}$, where $\beta := \beta_{\alpha}$ is given in~\cite[Corollary 5.12]{Wendland}.
\end{remark}

\begin{remark}
The constant $\gamma$ appearing in \eqref{eq:expDecayPhi} can be easily picked in some particular cases. 
For instance, for $\alpha = -1$, then $\gamma := \frac{\sqrt{2\pi}}{c}e^{\frac{1}{16 c \pi}}$. In this case, Theorem~\ref{thm:truncation} is satisfied for 
$\tau \geq 1 + \frac{1}{32 \pi^2 c^2} + \frac{\ln(\varepsilon^{-1})}{2\pi c} + \frac{\ln\left( \frac{2\sqrt{2\pi}}{c} \operatorname{cosh}(c\pi)\right)}{2\pi c} - \frac{\ln\left((1-e^{-2\pi c})D\right)}{2\pi c}$
% $\tau \geq 2 + \frac{\ln\left( \frac{2\sqrt{2\pi}}{c} \right)}{2\pi c} + \frac{1}{32 \pi^2 c^2} - \frac{\ln\left((1-e^{-2\pi c})D\right)}{2c\pi}$
\end{remark}

\subsection{Particular cases of Theorem~\ref{thm:truncation}}

\paragraph{Poisson kernel: case $\alpha = -1$} 
This case is associated to the approximation using a Poisson kernel as the basis function. 
The Bessel function involved can be simplified to $K_{-1/2}(r) = \sqrt{\frac{\pi}{2r}}e^{-r}$. 
Carrying out a similar analysis as in the proof of Theorem~\ref{thm:truncation} with $r_k = \xi^*+2\pi k \neq 0$, yields
\begin{align*}
\widehat{\phi_{-1,c}}(r_k) &= \frac{\pi}{c} e^{-c|r_k|},  \\
  &\leq \begin{cases}\frac{\pi}{c}e^{-c\xi^*}e^{-2\pi c k}, \quad \text{for } k > \tau  \\
   \frac{\pi}{c}e^{c\xi^*}e^{-2\pi c k}, \quad \text{for } k < -\tau.
  \end{cases}\end{align*}
Hence, the same argument as in the proof of Theorem \ref{thm:truncation} implies that $$
\sum_{|k| > \tau}\widehat{\phi_{1,c}}(\xi^*+2\pi k) \leq \frac{2\pi}{c}\frac{e^{-2\pi c (\tau +1)}}{1-e^{-2\pi c}}\operatorname{cosh}(c\xi^*).
$$
\begin{comment}Hence,
\begin{align*}
\sum_{k > \tau} \widehat{\phi_{-1,c}}(r_k) &\leq \frac{\pi}{c}\frac{e^{-2\pi c (\tau +1)}}{1-e^{-2\pi c}}e^{-c\xi^*}, \quad \text{ and } \\
\widehat{\phi_{-1,c}}(r_k) &\leq \frac{\pi}{c}\frac{e^{-2\pi c (\tau +1)}}{1-e^{-2\pi c}}e^{c\xi^*}, \\
\end{align*}
whereby,
$$
\sum_{|k| > \tau}\widehat{\phi_{1,c}}(\xi^*+2\pi k) \leq \frac{2\pi}{c}\frac{e^{-2\pi c (\tau +1)}}{1-e^{-2\pi c}}\operatorname{cosh}(c\xi^*).
$$
\end{comment}
We want now to ensure the condition 
$\frac{2\pi}{c}\frac{e^{-2\pi c (\tau +1)}}{1-e^{-2\pi c}}\operatorname{cosh}(c\xi^*) \leq \varepsilon D e^{-4\pi c}$ where $D$ is given by Remark \ref{Remark:D} as follows:
\begin{equation}
D_{-1} = \frac{1}{2c\sqrt{2}}.
\end{equation}
Finally, putting everything together, we obtain
\begin{equation}
\tau \geq \frac{\ln(\varepsilon^{-1})}{2\pi c} + 1 + \frac{\ln\left(\frac{4\pi \sqrt{2}\operatorname{cosh}(c\xi^*)}{1-e^{-2\pi c}}\right) }{2\pi c}
\end{equation}
ensures a relative error of the truncated series within $\varepsilon$, for any $\varepsilon > 0$. 
As an example, let us consider the accuracy of a single precision machine $\varepsilon = 10^{-16}$ and a shape parameter $c = 1$. 
In this case, $\tau \geq 7.712$
is sufficient (noting that $|\xi^*|< \pi$). 
In other words, only $17$
evaluations are required for an accurate estimation of the Fourier transform of the cardinal function. 
Similarly, for a double precision machine with $\varepsilon = 10^{-32}$, only $27$ coefficients in the expansion are sufficient for the accurate estimation of the Fourier transform.

\paragraph{Gaussian case:}
As mentioned before, every result stated here holds when the Gaussian kernel, $g_\lambda:=e^{-\lambda|\cdot|^2}$, is used.  Typically, one considers $0<\lambda\leq1$, and the limiting results above hold for $\lambda\to0^+$.  For simplicity, we omit the calculations for the Gaussian as they are rather similar to the Poisson kernel.  One finds that to obtain relative error $\eps$, one needs $\tau\geq \frac{2}{\pi^2}|\ln(\varepsilon/4)|+4$ terms, whence for machine precision, 12 terms are sufficient.

\begin{remark}
In the end, starting with \eqref{EQcardinaldef}, we are left with the numerical approximation of the cardinal function at a given point $x \in [-N,N]$. 
This can be done by first evaluating its (approximate) Fourier transform $\widehat{L_{\alpha,c}}(\xi)$ at the sampling points $\xi_k = k/(2N+1)$, for $-N \leq k \leq N.$ 
Following Theorem~\ref{thm:truncation}, it suffices to evaluate the Bessel functions at only few sampling points and sum them together. 
Finally, a direct application of a Discrete Fourier Transform allows for the computation of the cardinal function $L_{\alpha,c}$ at some (uniform) sampling points $(x_j)_{j \in \mathcal{J}}$ for some uniform finite set $\mathcal{J} \subset [-N,N]$ via 
\begin{equation*}
    x(k) = \sum_{n = - N}^N \widehat{L_{\alpha,c}}(\xi_n)e^{2\pi i kn/|\mathcal{J}|}.
\end{equation*}
Note that with the use of Fast Fourier Transform algorithms, this can be achieved with a total complexity of $\mathcal{O}(N\log(N))$. 
For a sampling set $\mathcal{J}$ large enough -- in other words, which samples the Fourier transform of the cardinal function $\widehat{L_{\alpha,c}}$ with enough detail -- we are able to build an accurate estimation of $L_{\alpha,c}$ at uniform points. 
Then one can interpolate to approximate the points in the expansion \eqref{EQcardinaldef} that have not been calculated exactly via the FFT. 
\end{remark}

\section{Remarks and Extensions}\label{SECExtensions}
%%%%%%%%%%%%%%%%%%%%%%%%%%%%%%%%%%%%%%%%%%%%%%%%%%%%%%%%
%%%%%%%%%%%%%%%%%%%%%%%%%%%%%%%%%%%%%%%%%%%%%%%%%%%%%%%%

We now conclude with some remarks on the previous results including comments on extensions to other kernels and higher dimensions and some considerations for future work.

\begin{remark}
There are some known convergence phenomena and approximation results for multivariate bandlimited and Sobolev interpolation in the vein above.  Specifically, \cite{BSS,HammZonotopes} contain higher dimensional analogues of Theorem \ref{THMScatteredRecovery}, while the uniform results in higher dimensions may be found throughout the work of Riemenschneider and Sivakumar.  Similarly, extensions to Theorem \ref{THMmeshApproximationRate} are discussed in \cite{HammJAT,schsiv}, though these are for specific CIS in higher dimensions which are Cartesian products of univariate ones.  Unfortunately, the multivariate case of  Theorem \ref{THM6.6} will likely prove more difficult for the reason that finding CISs in higher dimensions remains elusive even for simple domains.
\end{remark}

\begin{remark}
As mentioned often before, the use of the multiquadrics here exhibits a particular case of a much more general phenomenon. All of the results here will hold for Gaussians, $g_\lambda:=e^{-\lambda|\cdot|^2}$, $\lambda>0$, where in the limit one takes $\lambda\to0^+$.  In addition, the kernels $g_{\lambda,p}:=(e^{-\lambda|\cdot|^p})^\vee$, $\lambda>0$, $p>0$ will allow for recovery as $\lambda\to\infty$.  For sufficient conditions for a family of kernels governed by a parameter to allow for recovery in the sense of Theorem \ref{THMScatteredRecovery}, see \cite{LedfordScattered}, while \cite{LedfordCardinal} gives conditions for cardinal interpolation at the integer lattice.
\end{remark}

\begin{remark}
Section \ref{SECComputations} is meant to give some guideline for the use of cardinal functions in the sampling series; however, it is not meant to give a comprehensive treatment to computations with these cardinal functions.  Indeed, this is an avenue for further consideration because at the moment, it is not clear if the benefit to summability that replacing sinc with $\widehat{L_{\alpha,c}}$ gives is not offset by the extra computation required to evaluate $L_{\alpha,c}$.  Additionally, while we have argued theoretically that the use of cardinal functions for interpolation of time-limited Sobolev functions should have advantage over the naive RBF interpolation approach, more needs to be determined to demonstrate computational benefit.  In particular, an estimate of the constant $C$ from Theorem \ref{THM6.6} is needed, as well as some guarantees on estimating $L_{\alpha,c}$ at uniform points on $[-N,N]$.  Let us also point out that there are other ways around the poor condition number of the multiquadric and Guassian matrices involved in finding the interpolant for a compact domain; for a summary of some such methods, see Section 2 of \cite{HammLedford}, and additionally consult \cite{BuhmannBook,Wendland}.  \end{remark}

\end{document}